\documentclass{amsart}

\usepackage[lite]{amsrefs}
\usepackage{amssymb}
\usepackage{hyperref}

\DeclareMathOperator{\Q}{\mathbb{Q}}

\DeclareMathOperator{\Disc}{\mathrm{Disc}}
\DeclareMathOperator{\Gal}{\mathrm{Gal}}
\DeclareMathOperator{\Res}{\mathrm{Resultant}}
\DeclareMathOperator{\sgn}{\mathrm{sgn}}
\DeclareMathOperator{\N}{\mathcal{N}}

\numberwithin{equation}{section}

\theoremstyle{plain} %% This is the default, anyway
\newtheorem{theorem}{Theorem}[section]
\newtheorem{corollary}[theorem]{Corollary}
\newtheorem{lemma}[theorem]{Lemma}
\newtheorem{proposition}[theorem]{Proposition}

\begin{document}

\title{Galois groups of certain even octic polynomials}

\author{Malcolm Hoong Wai Chen} \address{Institute of Mathematical Sciences, Faculty of Science, University of Malaya, 50603 Kuala Lumpur, Malaysia.} \email{malcolmchen99@siswa.um.edu.my}
\author{Angelina Yan Mui Chin} \address{Institute of Mathematical Sciences, Faculty of Science, University of Malaya, 50603 Kuala Lumpur, Malaysia.} \email{acym@um.edu.my}
\author{Ta Sheng Tan} \address{Institute of Mathematical Sciences, Faculty of Science, University of Malaya, 50603 Kuala Lumpur, Malaysia.} \email{tstan@um.edu.my}


 \begin{abstract}
Let $f(x)=x^8+ax^4+b \in \Q[x]$ be an irreducible polynomial where $b$ is a square. We give a method that completely describes the factorization patterns of a linear resolvent of $f(x)$ using simple arithmetic conditions on $a$ and $b$. As a result, we determine the exact six possible Galois groups of $f(x)$ and completely classify all of them. As an application, we characterize the Galois groups of irreducible polynomials $x^8+ax^4+1 \in \Q[x]$. We also use similar methods to obtain analogous results for the Galois groups of irreducible polynomials $x^8+ax^6+bx^4+ax^2+1 \in \Q[x]$.
 \end{abstract}

\maketitle

\section{Introduction}

It is a classical result of abstract algebra that every polynomial over a field has a Galois group, which describes the permutation symmetries among the roots of the polynomial equation. While the Galois group of a given polynomial can be determined by computer algebra systems through the use of suitable resolvents \cite{stauduhar}, these calculations are computationally expensive. An interesting research problem is therefore to provide simple conditions on the roots (and hence, coefficients) of a polynomial for determination of its Galois group. It is well-known that determining the Galois group of a quartic polynomial generally requires factorization over a quadratic field. However, specializing to even quartic polynomials simplifies the algorithm into testing whether or not two quantities are squares, as described by the following result of Kappe and Warren \cite{kappewarren}.

\begin{proposition}[{\cite[Theorem 3]{kappewarren}}] \label{quarticgal}
Let $K$ be a field with characteristic not two and $f(x)=x^4+ax^2+b \in K[x]$ be irreducible. Then the Galois group of $f(x)$ is the
\begin{enumerate}
\item elementary abelian group of order four if $b$ is a square in $K$,
\item cyclic group of order four if $b$ is not a square in $K$ but $b(a^2-4b)$ is a square in $K$,
\item dihedral group of order eight if both $b$ and $b(a^2-4b)$ are not squares in $K$.
\end{enumerate}
\end{proposition}

A \emph{power compositional polynomial} is a polynomial of the form $f(x^d)$ for some monic $f(x) \in \Q[x]$ and positive integer $d \ge 2$. It is not coincidental that algorithms for determining Galois groups of power compositional polynomials are much simpler than general polynomials. Observe that $f(x)$ defines a subfield of the splitting field of $f(x^d)$, which in turn provides a certain subgroup structure for the Galois group of $f(x^d)$. This observation is used in \cite{awtreysextic} and \cite{awtreyoctic} to classify the Galois groups of $f(x^2)$ where $f(x)=x^3+ax^2+bx+c$ and $f(x)=x^4+ax^2+b$, respectively. Various other techniques are also used in \cite{polycomp1,polycomp2,polycomp3,polycomp4,polycomp5} to construct infinite families of power compositional polynomials having certain prescribed Galois groups.

\par Let $G_8$ and $G_4$ denote the Galois groups of irreducible polynomials $x^8+ax^4+b \in \Q[x]$ and $x^4+ax^2+b \in \Q[x]$, respectively. Altmann et al. \cite{awtreyoctic} described a systematic framework to determine the possible candidates of $G_8$ for a given $G_4$. They also proved that $G_8$ can be uniquely determined if either $G_4$ is the cyclic group of order four, or if $G_4$ is the dihedral group of order eight and $4b-a^2$ is a square. To the best of our knowledge, this is the only study on the Galois groups of power compositional octic polynomials.

\par The main aim of this paper is to classify and distinguish the Galois groups of certain irreducible power compositional octic polynomials. Our first focus is on \emph{doubly even octic polynomials} $x^8+ax^4+b \in \Q[x]$ where $G_4$ is the elementary abelian group of order four. In view of Proposition \ref{quarticgal},  this is equivalent to the fact that $b$ is a square, and is the only remaining case of $G_4$ that has yet to be studied in the literature. In Section \ref{prelim} we will establish several preliminary results to completely describe the factorization patterns of a linear resolvent using simple arithmetic conditions on $a$ and $b$, and describe the filtering framework introduced by Altmann et al. in \cite{awtreyoctic}. These preliminary results will be applied to prove our main results in Section \ref{resultde} where we determine the exact six possible Galois groups of doubly even octic polynomials (Theorem \ref{exactpossiblegal}) and then completely classify all of them (Theorem \ref{deoctics}). As an application, we completely characterize the Galois groups of $x^8+ax^4+1 \in \Q[x]$ (Corollary \ref{pdeoctics}). This leads us to our second focus on \emph{palindromic even octic polynomials} $x^8+ax^6+bx^4+ax^2+1 \in \Q[x]$ (where $a \neq 0$) in Section \ref{resultpe}. We extend the techniques in Sections 2 and 3 to determine the exact possible Galois groups of palindromic even octic polynomials (Proposition \ref{peocticspossible}) and provide a partial classification of these Galois groups (Theorems \ref{peocticse4} and \ref{peocticsc4}). To the best of our knowledge, this is the first study on the Galois groups of $f(x^2)$ where $f(x)$ is a quartic polynomial with all non-zero terms.

\par Throughout the paper we will use the following notations: $C_n$ is the cyclic group of order $n$, $E_n$ is the elementary abelian group of order $n$, $D_n$ is the dihedral group of order $2n$, $A_n$ is the alternating group on $n$ letters, $S_n$ is the symmetric group on $n$ letters, $\mathrm{nTj}$ is the $j$-th conjugacy class among transitive subgroups of $S_n$ according to Butler and McKay \cite{butlermckay}, $\Gal(L/K)$ is the Galois group of the field extension $L/K$, $\Gal_K(f)$ is the Galois group of the polynomial $f(x)$ over the field $K$, $\Disc_K(f)$ is the discriminant of the polynomial $f(x)$ over the field $K$, and $K^2$ is the set $\{k^2 : k \in K\}$ of square elements in $K$. We omit the subscript $K$ in $\Gal_K(f)$ and $\Disc_K(f)$ if the base field $K$ is clear from the context (usually taken to be $\Q$).

\section{Preliminary Results} \label{prelim}

\subsection{Irreducibility of power compositional octic polynomials}

It is a challenging task to determine the irreducibility for a general family of polynomials with symbolic coefficients. However, a complete yet simple characterization for the irreducibility of power compositional sextic polynomials had been determined by Harrington and Jones \cite{harringtonjones}. They provided sufficient and necessary conditions for a sextic polynomial to be reducible by describing the possible factorization patterns and relationships between the coefficients of the factors. This amounted to solving a system of equations, where its solutions describe the irreducible factors (and hence, irreducibility) of the sextic polynomial. A key component of this technique is the following theorem of Capelli, see \cite[Section 2.1]{schinzel}.

\begin{proposition} \label{capelli}
Let $f(x),g(x) \in \Q[x]$ where $f(x)$ is irreducible, and let $\alpha$ be a root of $f(x)$. Then $f(g(x))$ is reducible over $\Q$ if and only if $g(x)-\alpha$ is reducible over $\Q(\alpha)$. Furthermore, if
\begin{eqnarray*}
g(x)-\alpha = c_1 u_1(x)^{e_1} \cdots u_k(x)^{e_k} \in \Q(\alpha)[x]
\end{eqnarray*}
where $u_1(x),\dots,u_k(x)$ are distinct monic polynomials irreducible over $\Q(\alpha)$, then
\begin{eqnarray*}
f(g(x)) = c_2 \N(u_1(x))^{e_1} \cdots \N(u_k(x))^{e_k} \in \Q[x]
\end{eqnarray*}
where the norms $\N(u_1(x)),\dots,\N(u_k(x))$ are distinct monic polynomials irreducible over $\Q$.
\end{proposition}

We remark that this technique can be generalized to any power compositional polynomials $f(x^d)$ by choosing $g(x)=x^d$, and is particularly effective for small values of $d=2,3$ where the irreducibility of $g(x)-\alpha$ can be determined easily. We now characterize the irreducibility of power compositional octic polynomials, and then specialize it to doubly even octic polynomials.

\begin{proposition} \label{polycompeq}
Let $f(x)=x^4+ax^3+bx^2+cx+d \in \Q[x]$ be irreducible. Then $f(x^2)$ is reducible if and only if there exist $k,\ell,m,n \in \Q$ satisfying
\begin{equation} \label{p1}
a=2\ell -k^2, \ \ b=2n-2km+\ell^2, \ \ c=2\ell n-m^2, \ \ d=n^2.
\end{equation}
\end{proposition}

\begin{proof}
Choose some root $\alpha$ of $f(x)$. By Proposition \ref{capelli}, $f(x^2)$ is reducible if and only if $x^2-\alpha$ is reducible over $\Q(\alpha)$, that is, there exists some $\beta \in \Q(\alpha)$ such that $\alpha=\beta^2$. For such cases, $x^2-\alpha=(x+\beta)(x-\beta) \in \Q(\alpha)[x]$. Note that $x+\beta$ and $x-\beta$ are distinct monic polynomials irreducible over $\Q(\alpha)$. Let $\beta_1,\beta_2,\beta_3$ be conjugates of $\beta$ in the field extension $\Q(\alpha)/\Q$, and let $k,\ell,m,n \in \Q$ be the first, second, third and fourth elementary symmetric polynomials in $\beta,\beta_1,\beta_2,\beta_3$, respectively. Then
\begin{gather*}
\N(x+\beta) = (x+\beta)(x+\beta_1)(x+\beta_2)(x+\beta_3) = x^4+kx^3+\ell x^2+mx+n,
\\
\N(x-\beta) = (x-\beta)(x-\beta_1)(x-\beta_2)(x-\beta_3) = x^4-kx^3+\ell x^2-mx+n.
\end{gather*}
It follows that
\begin{eqnarray*}
&& x^8+ax^6+bx^4+cx^2+d \\
&=&(x^4+kx^3+\ell x^2+mx+n)(x^4-kx^3+\ell x^2-mx+n).
\end{eqnarray*}
By expanding the factors and comparing coefficients, we have (\ref{p1}). This proves the necessity. The sufficiency follows from construction.
\end{proof}
 
\begin{proposition} \label{polycompreducible}
Let $f(x)=x^4+ax^2+b \in \Q[x]$ be irreducible. Then $f(x^2)$ is reducible if and only if $b=r^4$ for some $r \in \Q$ and at least one of the following is in $\Q$:
\begin{equation} \label{p2}
\begin{split}
\sqrt{+4r+2\sqrt{2r^2+a}}, \quad \sqrt{-4r+2\sqrt{2r^2+a}}, \\
\sqrt{+4r-2\sqrt{2r^2+a}}, \quad \sqrt{-4r-2\sqrt{2r^2+a}}.
\end{split}
\end{equation} 
\end{proposition}

\begin{proof}
By Proposition \ref{polycompeq}, $f(x^2)$ is reducible if and only if the following system of equations in $k,\ell,m,n$ has a solution over $\Q$: 
\begin{align}
0=2\ell -k^2, \label{e1} \\ 
a=2n-2km+\ell^2, \label{e2} \\ 
0=2\ell n-m^2, \label{e3} \\ 
b=n^2. \label{e4}
\end{align}
If $n=0$, then $b=0$ and hence, $f(x)$ is reducible, a contradiction. As such, $n \neq 0$. Now multiplying (\ref{e1}) by $n$ and equating with (\ref{e3}), we have 
\begin{eqnarray}
nk^2=m^2. \label{e5}
\end{eqnarray} 
If either $k=0$ or $m=0$, then (\ref{e5}) implies that $k=m=0$. Substituting into (\ref{e1}) we have $\ell=0$. It follows from (\ref{e2}) that $a=2n$ and hence, $f(x)=x^4+2nx^2+n^2=(x^2+n)^2$ is reducible, a contradiction. Therefore, $k$ and $m$ are non-zero and $n=(m/k)^2 \in \Q^2$, that is, $n=r^2$ for some $r \in \Q$. It follows from (\ref{e4}) that $b=(r^2)^2=r^4$.
\par By (\ref{e1}) and (\ref{e5}), we have $\ell=k^2/2$ and $m=\pm\sqrt{n}k=\pm rk$, respectively. This implies that $\ell,m \in \Q$ if and only if $k \in \Q$. Substituting into (\ref{e2}), we have
\begin{eqnarray*}
a=2(r^2)-2k(\pm rk)+\left(\frac{k^2}{2}\right)^2.
\end{eqnarray*}
This can be rearranged into
\begin{eqnarray*}
\frac{1}{4}k^4 \mp 2rk^2 + 2r^2-a = 0
\end{eqnarray*}
which is a quartic equation in $k$ whose solutions are (\ref{p2}) and their negatives. Therefore, the system of equations (\ref{e1})--(\ref{e4}) has a solution over $\Q$ if and only if $b=r^4$ for some $r \in \Q$ and any of the expressions in (\ref{p2}) is in $\Q$. This proves the necessity. The sufficiency follows from construction.
\end{proof}

We will also need the following three propositions to determine the irreducibility of certain quartic and power compositional octic polynomials in Sections \ref{resultde} and \ref{resultpe}.

\begin{proposition} \label{polycompcor}
Let $f(x)=x^4+ax^3+bx^2+cx+d \in \Q[x]$ be an irreducible polynomial and suppose that $f(x^2)$ is reducible. Then $d=n^2$ for some $n \in \Q$ and there exists a rational solution to the equation
\begin{equation}
x^4-(2b+12n)x^2+(8c+8an)x+b^2-4ac-4bn+4n^2=0. \label{ll}
\end{equation}
\end{proposition}

\begin{proof}
This is a direct consequence of Proposition \ref{polycompeq} where we eliminate $k$ and $m$ to write $\ell$ in terms of $a,b,c,n$.
\end{proof}

\begin{proposition}[{\cite[Theorem 2]{kappewarren}}] \label{evenquarticreducible}
Let $K$ be a field with characteristic not two. Then $f(x)=x^4+ax^2+b \in K[x]$ is irreducible if and only if $a^2-4b,-a+2\sqrt{b},-a-2\sqrt{b} \notin K^2$.
\end{proposition}

\begin{proposition}[{\cite[Theorem 1]{brookfield}}] \label{quarticreducible}
The polynomial $x^4+cx^2+dx+e \in \Q[x]$ factors as the product of two quadratic polynomials if and only if at least one of the following holds:
\begin{enumerate}
\item The resolvent cubic $x^3+2cx^2+(c^2-4e)x-d^2$ has a non-zero root in $\Q^2$.
\item $d=0$ and $c^2-4e \in \Q^2$.
\end{enumerate}
\end{proposition}

\subsection{Linear resolvents and resultants}

We recall some important results on linear resolvents as described in \cite{soicher}.

\begin{proposition} \label{soicher}
Let $f(x) \in \Q[x]$ be an irreducible polynomial of degree $n$ and $\alpha_1,\dots,\alpha_n$ be all the roots of $f(x)$. Let $F(x_1,\dots,x_n) \in \Q[x_1,\dots,x_n]$ and let
\begin{equation*}
H=\{\sigma \in S_n : F(x_{\sigma(1)},\dots,x_{\sigma(n)}) = F(x_1,\dots,x_n)\}.
\end{equation*}
Then the resolvent polynomial of $f(x)$ corresponding to $F$ is
\begin{equation*}
R(x):=\prod_{\sigma \in S_n//H} (x-F(\alpha_{\sigma(1)},\dots,\alpha_{\sigma(n)}) \in \Q[x]
\end{equation*}
where $S_n//H$ is a complete set of right coset representatives of $H$ in $S_n$. The irreducible factors of $R(x)$ that occur with multiplicity one correspond to the orbits of the action of $\Gal(f)$ on the cosets in $S_n/H$, and the Galois group for any of these irreducible factors is the image of the permutation representation of this action on its corresponding coset. Furthermore, if $F[x_1,\dots,x_n]=x_1+\cdots+x_r$ where $r \leq n$, then the orbits of the action of $\Gal(f)$ on the cosets in $S_n/H$ correspond to the orbits of the action of $\Gal(f)$ on all the $r$-sets of $n$ letters.
\end{proposition}

In particular, we will use the linear resolvent corresponding to $F=x_1+x_2$ for our classification of Galois groups of $f(x)$. We remark that this linear resolvent can be computed using resultants as the polynomial $R(x)$ satisfying the relation 
\begin{equation}
R(x)^2 = \frac{\Res_y(f(y),f(x-y))}{2^8 \cdot f(x/2)}. \label{resultant}
\end{equation}

\subsection{Possible Galois groups} \label{filter}

We describe the filtering framework used in \cite{awtreyoctic} to obtain a preliminary list of possible Galois groups of irreducible polynomials $f(x)=g(x^2)$ based on the Galois group of the quartic polynomial $g(x)$. Let $K_f$ and $K_g$ be field extensions of $\Q$ obtained by adjoining a root of $f(x)$ and $g(x)$ respectively, and let $H_f$ and $H_g$ be subgroups of $\Gal(f)$ corresponding to $K_f$ and $K_g$ respectively under the Galois correspondence. Then $H_f$ is the point stabilizer of 1 under $\Gal(f)$ and $H_g$ is an index four subgroup of $\Gal(f)$ containing $H_f$. Furthermore, the splitting field of $g(x)$ is the normal closure of $K_g$. The list of Galois groups of the normal closures of non-isomorphic intermediate subfields of $K_f/\Q$ is said to be the \emph{subfield content of $f(x)$}, which is an invariant of $\Gal(f)$ \cite[Proposition 2.1]{awtreysubfield} and hence, can be determined completely by group-theoretic calculations. In particular, we have $\Gal(g)$ in the subfield content of $f(x)$. 
\par We also consider the discriminant of $f(x)$. Recall that $\Gal(f)$ is contained in $A_8$ if and only if $\Disc(f) \in \Q^2$. For power compositional polynomials, whether or not $\Disc(f) \in \Q^2$ can be easily checked by the following result.

\begin{proposition}[{\cite[Theorem 2.4]{awtreyoctic}}] \label{polycompdisc}
Let $f(x) \in \Q[x]$ be a monic polynomial with degree $m$, and let $c$ be the constant term of $f(x)$. Let $k$ be a positive integer and $n=mk$. Then
\begin{eqnarray*}
\Disc(f(x^k)) = (-1)^{n(n-m)/2} \cdot k^n \cdot c^{k-1} \cdot \Disc(f)^k.
\end{eqnarray*}
If $k$ is even, then $\Disc(f) \in \Q^2$ if and only if $(-1)^{n/2} c \in \Q^2$.
\end{proposition}

The filtering framework narrows down the list of possible Galois groups by providing us information on certain group-theoretic properties of the Galois groups. Recall that the Galois group of an irreducible polynomial is a transitive group, so we examine all 50 conjugacy classes of transitive subgroups of $S_8$ using the table of Galois groups from the database LMFDB \cite{lmfdb}. The \verb|Subfields| column of the table gives the subfield content of $f(x)$, whereas the \verb|Parity| column of the table records whether or not the group is contained in $A_8$. We remark that this framework can be generalized to any power compositional polynomials $f(x)$ by choosing $g(x)$ so that $f(x)=g(x^d)$.

\section{Doubly Even Octic Polynomials} \label{resultde}

\subsection{Factorization patterns of the linear resolvent}

Let $f(x)=x^8+ax^4+b \in \Q[x]$ be an irreducible polynomial where $b \in \Q^2$. We now use (\ref{resultant}) to compute the linear resolvent $R(x)$ of $f(x)$ corresponding to $F=x_1+x_2$. Computing this expression in Mathematica \cite{wmath} shows that the degree 28 resolvent $R(x)$ is the product of $x^4$ with three octic polynomials $R_1(x^2)$, $R_2(x^2)$ and $R_3(x^2)$ where
\begin{equation} \label{r1}
\begin{split}
R_1(x) &= x^4+(2a+12\sqrt{b})x^2+(a-2\sqrt{b})^2, \\ 
R_2(x) &= x^4+(2a-12\sqrt{b})x^2+(a+2\sqrt{b})^2, \\ 
R_3(x) &= x^4-4ax^2+16b. 
\end{split}
\end{equation}

In what follows, for brevity (with abuse of notation) we use $R_i$ to refer to all three polynomials $R_1,R_2,R_3$ defined in (\ref{r1}) collectively. We also remark that any irreducible factor of $R(x)$ that occurs with multiplicity one has degree at least four. This can be verified by computing the list of orbit lengths of the action of $G$ on all the 2-sets of eight letters for each transitive subgroup $G$ of $S_8$. In particular, any square-free degree four factor of $R(x)$ is irreducible.

\par We now determine the irreducibility and factorization patterns of $R_i(x^2)$, and show that each of the irreducible factors in $R_i(x^2)$ of the linear resolvent $R(x)$ occurs with multiplicity one so that we can apply Proposition \ref{soicher}. We begin by studying the irreducibility of $R_i(x)$.

\begin{table}[h]
\caption{The values of $A^2-4B$, $-A+2\sqrt{B}$ and $-A-2\sqrt{B}$ for the given polynomials $R_i(x)=x^4+Ax^2+B$. \label{table1}}
{\begin{tabular}{|c|c|c|c|} \hline 
\textbf{Polynomial} & $A^2-4B$ & $-A+2\sqrt{B}$ & $-A-2\sqrt{B}$ \\ \hline
$R_1(x)$ & $64(2b+a\sqrt{b})$ & $-16\sqrt{b}$ & $4(-a-2\sqrt{b})$ \\ \hline
$R_2(x)$ & $64(2b-a\sqrt{b})$ & $16\sqrt{b}$ & $4(-a+2\sqrt{b})$ \\ \hline
$R_3(x)$ & $16(a^2-4b)$ & $4(a+2\sqrt{b})$ & $4(a-2\sqrt{b})$ \\ \hline
\end{tabular}}
\end{table}

\begin{proposition} \label{linres2}
Define $R_1,R_2,R_3$ as in {\upshape (\ref{r1})}. Then
\begin{enumerate}
\item $R_1(x)$ is irreducible if and only if $2b+a\sqrt{b} \notin \Q^2$, 
\item $R_2(x)$ is irreducible if and only if $\sqrt{b} \notin \Q^2$ and $2b-a\sqrt{b} \notin \Q^2$,
\item $R_3(x)$ is irreducible if and only if $a-2\sqrt{b} \notin \Q^2$ and $a+2\sqrt{b} \notin \Q^2$.
\end{enumerate}
\end{proposition}

\begin{proof}
By Proposition \ref{evenquarticreducible} and Table \ref{table1}, $R_1(x)$ is irreducible if and only if $2b+a\sqrt{b},-a-2\sqrt{b},-\sqrt{b} \notin \Q^2$, $R_2(x)$ is irreducible if and only if $2b-a\sqrt{b},-a+2\sqrt{b},\sqrt{b} \notin \Q^2$, and $R_3(x)$ is irreducible if and only if $a^2-4b,a+2\sqrt{b},a-2\sqrt{b} \notin \Q^2$. Since $x^8+ax^4+b$ is irreducible, it follows that $x^4+ax^2+b$ is also irreducible and hence, $a^2-4b,-a+2\sqrt{b},-a-2\sqrt{b} \notin \Q^2$ by Proposition \ref{evenquarticreducible}. Since $-\sqrt{b}<0$, we also have $-\sqrt{b} \notin \Q^2$. This completes the proof. 
\end{proof}

We now study the irreducibility of $R_i(x^2)$. Recall that we have shown in Proposition \ref{polycompreducible} that for a power compositional octic polynomial $R_i(x^2)$ where $R_i(x)$ is an irreducible even quartic polynomial, $R_i(x^2)$ is irreducible if and only if certain quantities are irrational. We now show that these conditions are always satisfied. 

\begin{proposition} \label{linres1}
Define $R_1,R_2,R_3$ as in {\upshape (\ref{r1})}. Then for every $1 \leq i \leq 3$, $R_i(x^2)$ is irreducible if and only if $R_i(x)$ is irreducible.
\end{proposition} 

\begin{proof}
It is clear that if $R_i(x)$ is reducible then $R_i(x^2)$ is also reducible. Now suppose that $R_i(x)$ is irreducible. For $i=1$,  the solutions to $r^4=(a-2\sqrt{b})^2$ are $r=\pm\sqrt{a-2\sqrt{b}}$ and $r=\pm\sqrt{-a+2\sqrt{b}}$. Since $x^8+ax^4+b$ is irreducible, it follows that $x^4+ax^2+b$ is also irreducible and hence, $-a+2\sqrt{b} \notin \Q^2$ by Proposition \ref{evenquarticreducible}. Now if $\sqrt{a-2\sqrt{b}} \notin \Q$, then we are done. If not, then
\begin{eqnarray*}
&& \sqrt{\pm 4\left(\sqrt{a-2\sqrt{b}}\right) \pm 2\sqrt{2(a-2\sqrt{b})+(2a+12\sqrt{b})}} \\
&=& 2\sqrt{\pm \sqrt{a-2\sqrt{b}} \pm \sqrt{a+2\sqrt{b}}} \notin \Q,
\end{eqnarray*}
otherwise
\begin{equation*}
\sqrt{a^2-4b} = \left|\left(\sqrt{\pm \sqrt{a-2\sqrt{b}} \pm \sqrt{a+2\sqrt{b}}}\right)^4 - 2a\right| \in \Q
\end{equation*}
which implies that $x^2+ax+b$ is reducible, a contradiction to the fact that $x^8+ax^4+b$ is irreducible. Similarly for $i=2$, the solutions to $r^4=(a+2\sqrt{b})^2$ are $r=\pm\sqrt{a+2\sqrt{b}}$ and $r=\pm\sqrt{-a-2\sqrt{b}}$. Since $x^8+ax^4+b$ is irreducible, it follows that $x^4+ax^2+b$ is also irreducible and hence, $-a-2\sqrt{b} \notin \Q^2$ by Proposition \ref{evenquarticreducible}. Now if $\sqrt{a+2\sqrt{b}} \notin \Q$, then we are done. If not, then
\begin{eqnarray*}
&& \sqrt{\pm 4\left(\sqrt{a+2\sqrt{b}}\right) \pm 2\sqrt{2(a+2\sqrt{b})+(2a-12\sqrt{b})}} \\
&=& 2\sqrt{\pm \sqrt{a+2\sqrt{b}} \pm \sqrt{a-2\sqrt{b}}} \notin \Q.
\end{eqnarray*} 
For $i=3$, the solutions to $r^4=16b$ are $r=\pm2\sqrt{\sqrt{b}}$ and $r=\pm2\sqrt{\sqrt{-b}}$. Since $b \in \Q^2$, it follows that $-b \notin \Q^2$ and hence, $\sqrt{-b} \notin \Q^2$. Now if $\sqrt{\sqrt{b}} \notin \Q$, then we are done. If not, then
\begin{eqnarray*}
\sqrt{\pm 4\left(2\sqrt{\sqrt{b}}\right) \pm 2\sqrt{2(4\sqrt{b})+(-4a)}} = 2\sqrt{\pm 2\sqrt{\sqrt{b}} \pm \sqrt{-a+2\sqrt{b}}} \notin \Q,
\end{eqnarray*}
since
\begin{eqnarray*}
\pm 2\sqrt{\sqrt{b}} \pm \sqrt{-a+2\sqrt{b}} \notin \Q,
\end{eqnarray*}
otherwise $-a+2\sqrt{b} \in \Q^2$, which by Proposition \ref{evenquarticreducible} implies that $x^4+ax^2+b$ is reducible and hence, $x^8+ax^4+b$ is reducible, a contradiction.
\end{proof}

We can then use Proposition \ref{linres1} to give a complete factorization of $R_i(x^2)$.

\begin{proposition} \label{linresfactor}
Define $R_1,R_2,R_3$ as in {\upshape (\ref{r1})}. Then
\begin{enumerate}
\item $R_1(x^2)$ is irreducible if and only if $2b+a\sqrt{b} \notin \Q^2$. If $2b+a\sqrt{b} \in \Q^2$, then
\begin{eqnarray*}
R_1(x^2) = \left(x^4+a+6\sqrt{b}+4\sqrt{2b+a\sqrt{b}}\right)\left(x^4+a+6\sqrt{b}-4\sqrt{2b+a\sqrt{b}}\right).
\end{eqnarray*}
\item $R_2(x^2)$ is irreducible if and only if $\sqrt{b} \notin \Q^2$ and $2b-a\sqrt{b} \notin \Q^2$.
\begin{enumerate}
\item If $\sqrt{b} \in \Q^2$, then
\begin{eqnarray*}
R_2(x^2) = \left(x^4+4\sqrt{\sqrt{b}}x^2+a+2\sqrt{b}\right)\left(x^4-4\sqrt{\sqrt{b}}x^2+a+2\sqrt{b}\right).
\end{eqnarray*}
\item If $2b-a\sqrt{b} \in \Q^2$, then
\begin{eqnarray*}
R_2(x^2) = \left(x^4+a-6\sqrt{b}+4\sqrt{2b-a\sqrt{b}}\right)\left(x^4+a-6\sqrt{b}-4\sqrt{2b-a\sqrt{b}}\right).
\end{eqnarray*}
\end{enumerate}
\item $R_3(x^2)$ is irreducible if and only if $a-2\sqrt{b} \notin \Q^2$ and $a+2\sqrt{b} \notin \Q^2$.
\begin{enumerate}
\item If $a-2\sqrt{b} \in \Q^2$, then
\begin{eqnarray*}
R_3(x^2) = \left(x^4+2\sqrt{a-2\sqrt{b}}x^2-4\sqrt{b}\right)\left(x^4-2\sqrt{a-2\sqrt{b}}x^2-4\sqrt{b}\right).
\end{eqnarray*}
\item If $a+2\sqrt{b} \in \Q^2$, then
\begin{eqnarray*}
R_3(x^2) = \left(x^4+2\sqrt{a+2\sqrt{b}}x^2+4\sqrt{b}\right)\left(x^4-2\sqrt{a+2\sqrt{b}}x^2+4\sqrt{b}\right).
\end{eqnarray*}
\end{enumerate}
\end{enumerate}
Moreover, each of the irreducible factors in $R_i(x^2)$ occurs as a factor of $R(x)$ with multiplicity one. 
\end{proposition}

\begin{proof}
The conditions for irreducibility of polynomials $R_i(x^2)$ follow from Propositions \ref{linres2} and \ref{linres1}, whereas the factorization of $R_i(x^2)$ can be obtained by mimicking the algebraic calculations in the proof of Proposition \ref{evenquarticreducible} (see \cite[Theorem 2]{kappewarren}). Lastly, it is routine to verify that any two irreducible factors described above are identical if and only if $a=2\sqrt{b}$, $a=-2\sqrt{b}$ or $b=0$. Each of these possibilities implies that $x^8+ax^4+b$ is reducible, a contradiction.
\end{proof}

\subsection{Exact possible Galois groups}

We now apply the filtering framework used in \cite{awtreyoctic} (see Section \ref{filter}).

\begin{proposition} \label{possiblegal}
Let $f(x)=x^8+ax^4+b \in \Q[x]$ be an irreducible polynomial where $b \in \Q^2$. Then $\Gal(f)$ is 8T2, 8T3, 8T4, 8T5, 8T9, 8T11 or 8T22.
\end{proposition}

\begin{proof}
In view of Proposition \ref{quarticgal}, we have $\Gal(x^4+ax^2+b)$ is $E_4$. We examine all 50 conjugacy classes of transitive subgroups of $S_8$ as given in LMFDB \cite{lmfdb}, and only nine of them have $E_4$ in its subfield content: 8T2, 8T3, 8T4, 8T5, 8T9, 8T11, 8T21, 8T22 and 8T31. Furthermore, it follows from Proposition \ref{polycompdisc} that $\Disc(f) \in \Q^2$ and hence, $\Gal(f)$ is contained in $A_8$. This is true for only seven of the nine possibilities, namely, 8T2, 8T3, 8T4, 8T5, 8T9, 8T11 and 8T22.
\end{proof}

We now rule out the possibility of 8T5, the quaternion group of order eight, as Galois group of this polynomial family. We then give numerical examples for each of the remaining six possibilities. Our proof for the impossibility of 8T5 is motivated by \cite{dean} which characterized the subfield structure of the splitting field of a polynomial with quaternion Galois group. The proof requires the use of Proposition \ref{quarticgal} to compute Galois groups of even quartic polynomials over quadratic fields and so, we shall begin with a simple lemma characterizing squares in quadratic fields.

\begin{lemma} \label{quadsquare}
Let $d \in \Q \setminus \Q^2$ and $x \in \Q$. Then $x \in \Q(\sqrt{d})^2$ if and only if $x \in \Q^2$ or $dx \in \Q^2$.
\end{lemma}

\begin{proof}
Suppose that $x \in \Q(\sqrt{d})^2$. Then
\begin{eqnarray*}
x=(a+b\sqrt{d})^2=(a^2+db^2)+2ab\sqrt{d}
\end{eqnarray*}
for some $a,b \in \Q$. Since $x \in \Q$ and $\sqrt{d} \notin \Q$, we have $2ab=0$. This occurs if and only if $a=0$ or $b=0$. The case $b=0$ corresponds to $x=a^2 \in \Q^2$, whereas $a=0$ corresponds to $x=db^2$, that is, $dx=(db)^2 \in \Q^2$.
\end{proof}

\begin{theorem} \label{exactpossiblegal}
Let $f(x)=x^8+ax^4+b \in \Q[x]$ be an irreducible polynomial where $b \in \Q^2$. Then $\Gal(f)$ is 8T2, 8T3, 8T4, 8T9, 8T11 or 8T22. Moreover, each of these possibilities does occur.
\end{theorem}

\begin{proof}
Suppose to the contrary that $\Gal(f)$ is 8T5, the quaternion group of order eight. It follows from algebraic calculations that
\begin{equation*}
\theta=\sqrt{\frac{1}{2}\left(\sqrt{-a+2\sqrt{b}}+\sqrt{-a-2\sqrt{b}}\right)}
\end{equation*}
is a root of $f(x)$. Since the degree of the splitting field of $f(x)$ over $\Q$ is eight, it follows that the octic field $L=\Q(\theta)$ is the splitting field of $f(x)$. Recall that the quaternion group is the only group of order eight where each of its subgroups of order four is cyclic. As such, under the Galois correspondence, subgroups of $\Gal(f)$ corresponding to the subfields of $L$ with index four are cyclic. Let $p=-a+2\sqrt{b}$ and $q=-a-2\sqrt{b}$. Factoring $f(x)$ over $\Q(\sqrt{p})$, we have 
\begin{eqnarray*}
f(x)=\big(x^4+\sqrt{p}x^2+\sqrt{b}\big)\big(x^4-\sqrt{p}x^2+\sqrt{b}\big).
\end{eqnarray*}
Since the degree of $L$ over $\Q(\sqrt{p})$ is four, it follows that one of the degree four factors of $f(x)$, say $g(x)$, is the minimal polynomial of $L$ over $\Q(\sqrt{p})$. Note that $\Gal_{\Q(\sqrt{p})}(g)$ is the Galois group of the normal closure of $L/\Q(\sqrt{p})$, which is exactly $\Gal(L/\Q(\sqrt{p}))$, the subgroup of $\Gal(f)$ corresponding to $\Q(\sqrt{p})$. By Proposition \ref{quarticgal}, we have $q\sqrt{b} \in \Q(\sqrt{p})^2$, and it follows from Lemma \ref{quadsquare} that either $q\sqrt{b} \in \Q^2$ or $pq\sqrt{b} \in \Q^2$. Similarly, factoring $f(x)$ over $\Q(\sqrt{q})$, we have
\begin{eqnarray*}
f(x)=\big(x^4+\sqrt{q}x^2-\sqrt{b}\big)\big(x^4-\sqrt{q}x^2-\sqrt{b}\big).
\end{eqnarray*}
It follows that $-p\sqrt{b} \in \Q(\sqrt{q})^2$, that is, either $-p\sqrt{b} \in \Q^2$ or $-pq\sqrt{b} \in \Q^2$. Our assumption requires that the conditions on $\Q(\sqrt{p})$ and $\Q(\sqrt{q})$ hold true simultaneously. We now show that all four possible cases resulting from these two conditions are impossible.
\par First, note that it is impossible for $p=0$, $q=0$ or $b=0$, since each of these possibilities implies that $f(x)$ is reducible. Also note that $\Q^2$ is closed under multiplication and division.
\par \textbf{Case 1:} $pq\sqrt{b} \in \Q^2$ and $-pq\sqrt{b} \in \Q^2$. Then $-1 \in \Q^2$, a contradiction. 
\par \textbf{Case 2:} $q\sqrt{b} \in \Q^2$ and $-pq\sqrt{b} \in \Q^2$. Then $a-2\sqrt{b}=-p \in \Q^2$. By Proposition \ref{linresfactor}, the degree eight factor $R_3(x^2)$ of the linear resolvent $R(x)$ of $f(x)$ is reducible. However, if $\Gal(f)$ is 8T5, then all degree eight factors of the linear resolvent $R(x)$ are irreducible (see Section \ref{gremark}), a contradiction.
\par \textbf{Case 3:} $pq\sqrt{b} \in \Q^2$ and $-p\sqrt{b} \in \Q^2$. Then $a+2\sqrt{b}=-q \in \Q^2$. Likewise by Proposition \ref{linresfactor}, the degree eight factor $R_3(x^2)$ of the linear resolvent $R(x)$ of $f(x)$ is reducible. Similar to Case 2, we arrive at a contradiction.
\par \textbf{Case 4:} $q\sqrt{b} \in \Q^2$ and $-p\sqrt{b} \in \Q^2$. Then $-p/q \in \Q^2$ and hence, $-pq =(-p/q)q^2 \in \Q^2$, that is, $pq=-r^2$ for some $r \in \Q$. We also have $q(p-q)=4q\sqrt{b} \in \Q^2$, that is, $q(p-q)=s^2$ for some $s \in \Q$. Substituting $pq$ with $-r^2$ and rearranging, we have $q^2+r^2+s^2=0$ and hence, $q=0$, a contradiction.
\par Lastly, to show that each of these six possibilities does occur, note that $x^8+1$, $x^8-x^4+1$, $x^8+3x^4+1$, $x^8+2x^4+4$, $x^8+9$, $x^8+x^4+4$ are irreducible polynomials with Galois group 8T2, 8T3, 8T4, 8T9, 8T11 and 8T22, respectively. The same examples are also given in \cite{awtreyoctic}.
\end{proof}

\subsection{Classification of Galois groups} \label{gremark}

Let $f(x)=x^8+ax^4+b \in \Q[x]$ be an irreducible polynomial where $b \in \Q^2$. For each of the possible candidates $G=\Gal(f)$ obtained in Theorem \ref{exactpossiblegal}, we determine the list of degrees of irreducible factors of the linear resolvent $R(x)$. This is equivalent to the list of orbit lengths for the action of $G$ on all 2-sets of eight letters. We implement this in GAP \cite{gap} by the command \verb|OrbitLengthsDomain(G,Combinations([1..8],2),OnSets)|, and record the output in Table \ref{table2}.

\begin{table}[h]
\caption{Possible Galois groups $G$ of $f(x)=x^8+ax^4+b$ (where $b \in \Q^2$) and list of orbit lengths for the action of $G$ on all 2-sets of eight letters. Multiplicative notation is used to denote multiplicity. \label{table2}}
{\begin{tabular}{|c|c|c|c|c|c|c|} \hline 
$\mathbf{G}$ & 8T2 & 8T3 & 8T4 & 8T9 & 8T11 & 8T22 \\ \hline
\textbf{Orbits} & $4^3,8^2$ & $4^7$ & $4^5,8$ & $4^3,8^2$ & $4,8^3$ & $4,8^3$ \\ \hline
\end{tabular}}
\end{table}

We are now ready to describe an algorithm to completely classify the Galois groups of doubly even octic polynomials $f(x)=x^8+ax^4+b \in \Q[x]$, where $b \in \Q^2$. Our classification mainly relies on the factorization patterns of the three degree eight factors $R_i(x^2)$ of the linear resolvent $R(x)$. We have shown in Proposition \ref{linresfactor} that the irreducible factors in $R_i(x^2)$ occur as a factor of $R(x)$ with multiplicity one. It follows from Table \ref{table2} that $\Gal(f)$ is 8T3 if all three of $R_i(x^2)$ are reducible, $\Gal(f)$ is 8T4 if exactly two of $R_i(x^2)$ are reducible, $\Gal(f)$ is either 8T2 or 8T9 if exactly one of $R_i(x^2)$ is reducible, $\Gal(f)$ is either 8T11 or 8T22 if none of $R_i(x^2)$ is reducible. We remark that the list of orbit lengths for the action of 8T5 on all 2-sets of eight letters is $4,8^3$, which results in a contradiction in Cases 2 and 3 in our proof of Theorem \ref{exactpossiblegal}.

\par To distinguish between 8T2 and 8T9, we first determine the three length four orbits for the action of $G \in \{\text{8T2},\text{8T9}\}$ on the 2-sets of eight letters. For each of these orbits $O$, we determine the image of the permutation representation of $G$ acting on $O$. These can be implemented in GAP \cite{gap} using the commands \verb|Orbits(G,Combinations([1..8],2),OnSets)| and \verb|Action(G,O,OnSets)|. For 8T2, the three orbits correspond to $E_4,C_4,C_4$ whereas for 8T9, the three orbits correspond to $E_4,D_4,D_4$. For each case, we show that $R(x)$ has a factor of the form $(x^4+Ax^2+B)(x^4-Ax^2+B)$. These two degree four factors have the same Galois group. Since the degree four factors occur in $R(x)$ with multiplicity one, it follows from Proposition \ref{soicher} that the Galois group of these factors corresponds to the repeating output. Therefore, we can distinguish between 8T2 and 8T9 by using Proposition \ref{quarticgal} to compute the Galois group of the degree four factor.

\par Recall that 8T11 and 8T22 have orders 16 and 32, respectively, so we can distinguish them by calculating the degree of the splitting field of $f(x)$ over $\Q$. Let $\theta$ be a root of $f(x)$. Factoring $f(x)$ over $\Q(\theta)$, we have $f(x)=(x-\theta)(x+\theta)f_1(x)f_2(x)f_3(x)$ where
\begin{eqnarray*}
f_1(x)=x^2+\theta^2, \quad f_2(x)=x^2-\frac{\sqrt{b}}{\theta^2}, \quad f_3(x)=x^2+\frac{\sqrt{b}}{\theta^2}.
\end{eqnarray*}

Observe that the roots of $f_1(x)$, $f_2(x)$ and $f_3(x)$ are 
\begin{eqnarray*}
\pm\sqrt{-1}\theta, \quad \pm\frac{\sqrt{\sqrt{b}}}{\theta}, \quad \pm\frac{\sqrt{-1}\sqrt{\sqrt{b}}}{\theta},
\end{eqnarray*}
respectively. It follows that the splitting field of $f(x)$ is the biquadratic extension $L=K(\sqrt{-1},\sqrt{\sqrt{b}})$ of $K=\Q(\theta)$, and so
\begin{equation*}
|\Gal(f)|=[L:\Q]=[L:K][K:\Q]=8[L:K] \leq 8(4) = 32,
\end{equation*}
with equality if and only if none of $-1,\sqrt{b},-\sqrt{b}$ is a square in $\Q(\theta)$. As such, the Galois group of $f(x)$ is 8T11 if any of these quantities is a square in $\Q(\theta)$, and 8T22 otherwise. This motivated the following proposition.

\begin{proposition} \label{octicsquares}
Let $f(x)=x^8+ax^4+b \in \Q[x]$ be an irreducible polynomial where $\sqrt{b} \in \Q \setminus \Q^2$, and let $\theta$ be a root of $f(x)$. Then for each $r \in \{-1,\sqrt{b},-\sqrt{b}\}$, we have $r \in \Q(\theta)^2$ if and only if at least one of $r(a^2-4b)$, $r(-a+2\sqrt{b})$ or $r(-a-2\sqrt{b})$ is in $\Q^2$.
\end{proposition}

\begin{proof}
Suppose that $r \in \Q(\theta)^2$. Then
\begin{eqnarray*}
\sqrt{r} &=& a_7 \theta^7 + a_6 \theta^6 + a_5 \theta^5 + a_4 \theta^4 + a_3 \theta^3 + a_2 \theta^2 + a_1 \theta + a_0.
\end{eqnarray*}
for some $a_0,a_1,a_2,a_3,a_4,a_5,a_6,a_7 \in \Q$. We first claim that we must have $a_1=a_3=a_5=a_7=0$. Suppose to the contrary that this is not the case. Since $f(x)$ is irreducible, $\Gal(f)$ acts transitively on the roots of $f(x)$. In particular, there is a mapping $\sigma \in \Gal(f)$ satisfying $\sigma(\theta)=-\theta$. It follows that
\begin{eqnarray*}
\sigma(\sqrt{r}) = a_6 \theta^6 + a_4 \theta^4 + a_2 \theta^2 + a_0 - (a_7 \theta^7 + a_5 \theta^5 + a_3 \theta^3 + a_1 \theta),
\end{eqnarray*}
and so $\sigma(\sqrt{r}) \neq \sqrt{r}$. By construction, $\sigma(\sqrt{r})$ is another root of the minimal polynomial of $\sqrt{r}$ over $\Q$. Since $r \in \Q$, this minimal polynomial is $x^2-r$ and hence, $\sigma(\sqrt{r})=-\sqrt{r}$, implying that $a_0=a_2=a_4=a_6=0$. 
\par Now pick $\theta'=\sqrt{r}\theta$ or $\theta'=\frac{\sqrt{r}}{\theta}$ such that $\theta'$ is a root of $f(x)$. Note that $\theta' \in \Q(\theta^2)$ and hence, $\Q(\theta') \subset \Q(\theta^2)$. However, $\theta^2$ is a root of the irreducible polynomial $x^4+ax^2+b$, so $[\Q(\theta'):\Q]=8>4=[\Q(\theta^2):\Q]$, a contradiction. This establishes our claim and hence, $\sqrt{r} = a_6 \theta^6 + a_4 \theta^4 + a_2 \theta^2 + a_0$. 
\par We now further claim that either $a_0=a_4=0$ or $a_2=a_6=0$. Suppose to the contrary that this is not the case. Since $\Gal(f)$ acts transitively on the roots of $f(x)$, there is a mapping $\tau \in \Gal(f)$ satisfying $\tau(\theta)=i\theta$. It follows that
\begin{eqnarray*}
\tau(\sqrt{r}) = a_4 \theta^4 + a_0 - (a_6 \theta^6 + a_2 \theta^2),
\end{eqnarray*}
and so $\tau(\sqrt{r}) \neq \sqrt{r}$. Similarly, $\tau(\sqrt{r})$ is another root of the minimal polynomial of $\sqrt{r}$ over $\Q$. Since $r \in \Q$, this minimal polynomial is $x^2-r$ and hence, $\tau(\sqrt{r})=-\sqrt{r}$. It follows that $a_0=a_4=0$, a contradiction.
\par\textbf{Case 1:} $a_2=a_6=0$. Then
\begin{eqnarray*}
r = (a_4 \theta^4 + a_0)^2 &=& a_4^2 \theta^8 + 2a_0 a_4 \theta^4 + a_0^2 \\
&=& a_4^2 \left(-(a \theta^4 + b)\right) + 2a_0 a_4 \theta^4 + a_0^2 \\
&=& (2a_0 a_4 - a a_4^2)\theta^4 + (a_0^2 - b a_4^2).
\end{eqnarray*}
Since $r \in \Q$, the coefficient of $\theta^4$ is zero; that is, $2a_0 a_4 - a a_4^2 = 0$. Now if $a_4=0$, then $r \in \Q^2$, a contradiction. As such $a_4 \neq 0$ and we have $a_0=\frac{a a_4}{2}$. It follows that
\begin{eqnarray*}
r = a_0^2 - b a_4^2 = \left(\frac{a a_4}{2}\right)^2 - b a_4^2 = \frac{a_4^2 (a^2-4b)}{4} = \left(\frac{a_4}{2}\right)^2 (a^2-4b).
\end{eqnarray*}
This implies that $r(a^2-4b) \in \Q^2$.
\par\textbf{Case 2:} $a_0=a_4=0$. Then
\begin{eqnarray*}
r = (a_6 \theta^6 + a_2 \theta^2)^2 &=& a_6^2 \theta^{12} + 2a_2 a_6 \theta^8 + a_2^2 \theta^4 \\
&=& a_6^2 \theta^4 \left(-(a \theta^4 + b)\right) + 2a_2 a_6 \left(-(a \theta^4 + b)\right) + a_2^2 \theta^4 \\
&=& -a a_6^2 \theta^8 + (a_2^2-2a a_2 a_6-b a_6^2)\theta^4 - 2 b a_2 a_6 \\
&=& -a a_6^2 \left(-(a \theta^4 + b)\right) + (a_2^2-2a a_2 a_6-b a_6^2)\theta^4 - 2 b a_2 a_6 \\
&=& (a_2^2-2a a_2 a_6+a^2 a_6^2-b a_6^2)\theta^4 + (ab a_6^2 - 2 b a_2 a_6).
\end{eqnarray*}
Since $r \in \Q$, the coefficient of $\theta^4$ is zero. Now if $a_6=0$, then $a_2=0$ and hence, $r=0 \in \Q^2$, a contradiction. As such, $a_6 \neq 0$ and solving as a quadratic of $a_2$ we have $a_2=a_6(a+\sqrt{b})$ or $a_2=a_6(a-\sqrt{b})$. Now if $a_2=a_6(a+\sqrt{b})$, then
\begin{eqnarray*}
r = ab a_6^2 - 2 b a_2 a_6 &=& ab a_6^2 - 2 b (a_6(a+\sqrt{b})) a_6 \\
&=& b a_6^2 (-a-2\sqrt{b}) \\
&=& (a_6 \sqrt{b})^2 (-a-2\sqrt{b}).
\end{eqnarray*}
Finally, if $a_2=a_6(a-\sqrt{b})$, then
\begin{eqnarray*}
r = ab a_6^2 - 2 b a_2 a_6 &=& ab a_6^2 - 2 b (a_6(a-\sqrt{b})) a_6 \\
&=& b a_6^2 (-a+2\sqrt{b}) \\
&=& (a_6 \sqrt{b})^2 (-a+2\sqrt{b}).
\end{eqnarray*}
This implies that either $r(-a+2\sqrt{b}) \in \Q^2$ or $r(-a-2\sqrt{b}) \in \Q^2$. This proves the necessity. The sufficiency follows from construction.
\end{proof}

Consequently, we have the following algorithm as our main result of this section.

\begin{theorem} \label{deoctics}
Let $f(x)=x^8+ax^4+b \in \Q[x]$ be an irreducible polynomial where $b \in \Q^2$. Then the following algorithm returns $\Gal(f)$.
\begin{enumerate}
\item If $\sqrt{b} \in \Q^2$, then
\begin{enumerate}
\item If $a+2\sqrt{b} \in \Q^2$, return 8T3 and terminate.
\item Else if $a-2\sqrt{b} \in \Q^2$, return 8T4 and terminate.
\item Else if $4b-a^2 \in \Q^2$, return 8T2 and terminate.
\item Otherwise, return 8T9 and terminate.
\end{enumerate}
\item Else if $\sqrt{b} \notin \Q^2$, then
\begin{enumerate}
\item If $a+2\sqrt{b} \in \Q^2$, then
\begin{enumerate}
\item If $2b-a\sqrt{b} \in \Q^2$, return 8T4 and terminate.
\item Else if $-(2b-a\sqrt{b}) \in \Q^2$, return 8T2 and terminate.
\item Otherwise, return 8T9 and terminate.
\end{enumerate}
\item Else if $a-2\sqrt{b} \in \Q^2$, then
\begin{enumerate}
\item If $2b+a\sqrt{b} \in \Q^2$, return 8T4 and terminate.
\item Else if $-(2b+a\sqrt{b}) \in \Q^2$, return 8T2 and terminate.
\item Otherwise, return 8T9 and terminate.
\end{enumerate}
\item Else if $a+2\sqrt{b} \notin \Q^2$ and $a-2\sqrt{b} \notin \Q^2$, then
\begin{enumerate} \itemsep0em
\item If $2b-a\sqrt{b} \in \Q^2$ and $2b+a\sqrt{b} \in \Q^2$, return 8T4 and terminate.
\item Else if either $2b-a\sqrt{b} \in \Q^2$ or $2b+a\sqrt{b} \in \Q^2$, return 8T9 and terminate.
\item Else if $4b-a^2 \in \Q^2$, $\sqrt{b}(a^2-4b) \in \Q^2$, $\sqrt{b}(4b-a^2) \in \Q^2$, $-(2b-a\sqrt{b}) \in \Q^2$ or $-(2b+a\sqrt{b}) \in \Q^2$, return 8T11 and terminate.
\item Otherwise, return 8T22 and terminate.
\end{enumerate}
\end{enumerate}
\end{enumerate}
\end{theorem}

\begin{proof}
Recall from Proposition \ref{linresfactor} that 
\begin{itemize} \itemsep0em
\item $R_1(x^2)$ is irreducible if and only if $2b+a\sqrt{b} \notin \Q^2$, 
\item $R_2(x^2)$ is irreducible if and only if $\sqrt{b} \notin \Q^2$ and $2b-a\sqrt{b} \notin \Q^2$,
\item $R_3(x^2)$ is irreducible if and only if $a-2\sqrt{b} \notin \Q^2$ and $a+2\sqrt{b} \notin \Q^2$.
\end{itemize}
\textbf{Case 1:} $\sqrt{b} \in \Q^2$. Then $R_2(x^2)$ is reducible. Observe that $2b+a\sqrt{b}=(a+2\sqrt{b})\sqrt{b} \in \Q^2$ if and only if $a+2\sqrt{b} \in \Q^2$. It follows that if $a+2\sqrt{b} \in \Q^2$, then both $R_1(x^2)$ and $R_3(x^2)$ are reducible, which proves $(1)(a)$. Now if $a+2\sqrt{b} \notin \Q^2$, then $R_1(x^2)$ is irreducible, and $R_3(x^2)$ is reducible if and only if $a-2\sqrt{b} \in \Q^2$, which proves $(1)(b)$. Now if $a-2\sqrt{b} \notin \Q^2$, note by Proposition \ref{linresfactor} that
\begin{eqnarray*}
R_2(x^2) = \left(x^4+4\sqrt{\sqrt{b}}x^2+a+2\sqrt{b}\right)\left(x^4-4\sqrt{\sqrt{b}}x^2+a+2\sqrt{b}\right).
\end{eqnarray*}
By Proposition \ref{quarticgal}, both degree four factors have Galois group $C_4$ if $4b-a^2 \in \Q^2$ and $D_4$ otherwise. This proves $(1)(c)$ and $(1)(d)$.
\par \textbf{Case 2:} $\sqrt{b} \notin \Q^2$ and $a+2\sqrt{b} \in \Q^2$. Then $R_3(x^2)$ is reducible. Observe that $2b+a\sqrt{b}=(a+2\sqrt{b})\sqrt{b} \notin \Q^2$ and hence, $R_1(x^2)$ is irreducible. Also observe that $R_2(x^2)$ is reducible if and only if $2b-a\sqrt{b} \in \Q^2$, which proves $(2)(a)(i)$. Now if $2b-a\sqrt{b} \notin \Q^2$, note by Proposition \ref{linresfactor} that
\begin{eqnarray*}
R_3(x^2) = \left(x^4+2\sqrt{a+2\sqrt{b}}x^2+4\sqrt{b}\right)\left(x^4-2\sqrt{a+2\sqrt{b}}x^2+4\sqrt{b}\right).
\end{eqnarray*}
By Proposition \ref{quarticgal}, both degree four factors have Galois group $C_4$ if $-(2b-a\sqrt{b}) \in \Q^2$ and $D_4$ otherwise. This proves $(2)(a)(ii)$ and $(2)(a)(iii)$.
\par \textbf{Case 3:} $\sqrt{b} \notin \Q^2$ and $a-2\sqrt{b} \in \Q^2$. Then $R_3(x^2)$ is reducible. Observe that $2b-a\sqrt{b}=(a-2\sqrt{b})(-\sqrt{b}) \notin \Q^2$ and hence, $R_2(x^2)$ is irreducible. Also observe that $R_1(x^2)$ is reducible if and only if $2b+a\sqrt{b} \in \Q^2$, which proves $(2)(b)(i)$. Now if $2b+a\sqrt{b} \notin \Q^2$, note by Proposition \ref{linresfactor} that
\begin{eqnarray*}
R_3(x^2) = \left(x^4+2\sqrt{a-2\sqrt{b}}x^2-4\sqrt{b}\right)\left(x^4-2\sqrt{a-2\sqrt{b}}x^2-4\sqrt{b}\right).
\end{eqnarray*}
By Proposition \ref{quarticgal}, both degree four factors have Galois group $C_4$ if $-(2b+a\sqrt{b}) \in \Q^2$ and $D_4$ otherwise. This proves $(2)(b)(ii)$ and $(2)(b)(iii)$.
\par \textbf{Case 4:} $\sqrt{b} \notin \Q^2$, $a+2\sqrt{b} \notin \Q^2$ and $a-2\sqrt{b} \notin \Q^2$. Then $R_3(x^2)$ is irreducible, and $R_1(x^2)$ is reducible if and only if $2b+a\sqrt{b} \in \Q^2$, whereas $R_2(x^2)$ is reducible if and only if $2b-a\sqrt{b} \in \Q^2$. This proves $(2)(c)(i)$ and $(2)(c)(ii)$. Finally, $(2)(c)(iii)$ and $(2)(c)(iv)$ follow from Proposition \ref{octicsquares}.
\end{proof}

Specifically, item $(1)$ in Theorem \ref{deoctics} provides a complete classification for the Galois groups of doubly even octic polynomials where $b$ is a fourth power. In particular, for the case $b=1$ we have the following result.

\begin{corollary} \label{pdeoctics}
Let $f(x)=x^8+ax^4+1 \in \Q[x]$ be irreducible. Then $\Gal(f)$ is
\begin{enumerate}
\item 8T3 if $a+2 \in \Q^2$,
\item 8T4 if $a+2 \not\in \Q^2$ and $a-2 \in \Q^2$,
\item 8T2 if $a+2 \not\in \Q^2$, $a-2 \not\in \Q^2$ and $4-a^2 \in \Q^2$,
\item 8T9 if $a+2 \not\in \Q^2$, $a-2 \not\in \Q^2$ and $4-a^2 \not\in \Q^2$.
\end{enumerate}
\end{corollary}

We remark that these simple arithmetic conditions can be used to generate parametric families of doubly even octic polynomials for each of the Galois groups 8T2, 8T3, 8T4 and 8T9. For example, the polynomial $f_t(x)=x^8+(t^2-2)x^4+1$ has Galois group 8T3 for any value of $t \in \Q$ where $f_t(x)$ is irreducible. These values of $t$ can in turn be determined using Proposition \ref{polycompreducible} which characterizes the irreducibility of doubly even octic polynomials.

\section{Palindromic Even Octic Polynomials} \label{resultpe}

\subsection{Factorization patterns of the linear resolvent}

A natural question that follows from Theorem \ref{deoctics} and Corollary \ref{pdeoctics} is whether similar methods can be used to study the Galois groups of palindromic even octic polynomials $f(x)=x^8+ax^6+bx^4+ax^2+1$ (where $a \neq 0$). In a similar fashion, we compute the linear resolvent $R(x)$ of irreducible polynomials $f(x)$ corresponding to $F=x_1+x_2$ using (\ref{resultant}). Computing this expression in Mathematica \cite{wmath} shows that the degree 28 resolvent $R(x)$ is the product of $x^4$ with a degree 16 polynomial $R_{16}(x)$ (see Appendix \ref{aa}) and two irreducible quartic polynomials 
\begin{equation} \label{r2}
\begin{split}
R_1(x)=x^4+(a-4)x^2+(b+2-2a), \\
R_2(x)=x^4+(a+4)x^2+(b+2+2a).
\end{split}
\end{equation}

In view of Proposition \ref{soicher}, we require that the irreducible factors occur with multiplicity one.

\begin{proposition} \label{pmult}
Define $R_1,R_2$ as in {\upshape (\ref{r2})}. Then each of the irreducible factors $R_1(x)$ and $R_2(x)$ of $R(x)$ occurs with multiplicity one.
\end{proposition}

\begin{proof}
Suppose to the contrary that $R_1(x)$ does not occur as an irreducible factor of $R(x)$ with multiplicity one. Clearly $R_1(x) \neq R_2(x)$ since their coefficients of $x^2$ are not equal. It follows that $R_1(x)$ is a factor of $R_{16}(x)$. Let
\begin{gather*}
t_1(u)=-4(u-6)(4u^2-16u+78-13b), \\
t_2(u)=32u^3-(263+16b)u^2+(1320+36b)u+40b^2-468b-1080.
\end{gather*} 
We use Mathematica \cite{wmath} to perform symbolic polynomial division on $R_{16}(x)$ by $R_1(x)$, which results in the remainder $(a^2-4b+8) \big(t_1(a)x^2+t_2(a)\big)$. Our assumption requires that the remainder be identically zero, which we now show to be impossible.
\par \textbf{Case 1:} $a^2-4b+8=0$. Then $b=\frac{a^2+8}{4}$ and hence,
\begin{eqnarray*}
f(x)=x^8+ax^6+\frac{1}{4}(a^2+8)x^4+ax^2+1=\left(x^4+\frac{1}{2}ax^2+1\right)^2
\end{eqnarray*}
is reducible, a contradiction.
\par \textbf{Case 2:} $t_1(a)=t_2(a)=0$. The rational solutions for this system of equations are
\begin{eqnarray*}
(a,b) \in \left\{(4,6),\left(6,\frac{51}{5}\right),\left(6,\frac{21}{2}\right),\left(\frac{221}{24},\frac{2989}{144}\right)\right\}.
\end{eqnarray*}
Factoring using a computer algebra system shows that $f(x)$ is reducible for these values of $a$ and $b$, a contradiction.
\par Likewise, suppose to the contrary that $R_2(x)$ does not occur as an irreducible factor of $R(x)$ with multiplicity one. Then $R_2(x)$ is a factor of $R_{16}(x)$, and the remainder when $R_{16}(x)$ is divided by $R_2(x)$ is $(a^2-4b+8) \big(-t_1(-a)x^2+t_2(-a)\big)$. We have shown that $a^2-4b+8 \neq 0$, so our assumption requires that $-t_1(-a)=t_2(-a)=0$. The rational solutions for this system of equations are
\begin{eqnarray*}
(a,b) \in \left\{(-4,6),\left(-6,\frac{51}{5}\right),\left(-6,\frac{21}{2}\right),\left(-\frac{221}{24},\frac{2989}{144}\right)\right\}.
\end{eqnarray*}
However, $f(x)$ is reducible for these values of $a$ and $b$, a contradiction.
\end{proof}

\subsection{Exact possible Galois groups}

We first identify the possible Galois groups of $f(x)=x^4+ax^3+bx^2+ax+1$ by using the classification in \cite[Section 10 (Exercise 7)]{kaplansky}. It can be derived as a corollary of \cite[Section 10 (Theorem 43)]{kaplansky}, so we do not reproduce the proof here.

\begin{proposition} \label{pocticsgal}
Let $f(x)=x^4+ax^3+bx^2+ax+1 \in \Q[x]$ be irreducible. Then $\Gal(f)$ is
\begin{enumerate}
\item $E_4$ if $(b+2)^2-4a^2 \in \Q^2$,
\item $C_4$ if $(b+2)^2-4a^2 \notin \Q^2$ and $(a^2-4b+8)\big((b+2)^2-4a^2\big) \in \Q^2$,
\item $D_4$ if $(b+2)^2-4a^2 \notin \Q^2$ and $(a^2-4b+8)\big((b+2)^2-4a^2\big) \notin \Q^2$.
\end{enumerate}
\end{proposition}

We now apply the filtering framework used in \cite{awtreyoctic} (see Section \ref{filter}).

\begin{proposition} \label{peocticspossible}
Let $f(x)=x^8+ax^6+bx^4+ax^2+1 \in \Q[x]$ (where $a \neq 0$) be irreducible, and let $g(x)=x^4+ax^3+bx^2+ax+1$ so that $f(x)=g(x^2)$.
\begin{enumerate}
\item If $\Gal(g)$ is $E_4$, then $\Gal(f)$ is 8T2, 8T3, 8T4 or 8T9.
\item If $\Gal(g)$ is $C_4$, then $\Gal(f)$ is 8T2 or 8T10.
\item If $\Gal(g)$ is $D_4$, then $\Gal(f)$ is 8T4, 8T9, 8T10 or 8T18.
\end{enumerate}
Moreover, each of the possibilities above does occur.
\end{proposition}

\begin{proof}
We examine all 50 conjugacy classes of transitive subgroups of $S_8$ as given in \cite{lmfdb}. Since the constant term of $f(x)$ is $1 \in \Q^2$, it follows from Proposition \ref{polycompdisc} that $\Gal(f)$ is contained in $A_8$. This allows us to filter a list of possible Galois groups in Table \ref{table3} where for each of the possible $\Gal(g)$, the possible $\Gal(f)$ has $\Gal(g)$ in the subfield content of $f(x)$, and is contained in $A_8$. For each of the possible candidates $G=\Gal(f)$, we determine the list of degrees of irreducible factors of the linear resolvent $R(x)$. This is equivalent to the list of orbit lengths for the action of $G$ on all 2-sets of eight letters. We implement this in GAP \cite{gap} by the command \verb|OrbitLengthsDomain(G,Combinations([1..8],2),OnSets)|, and record the output in Table \ref{table4}. Since the linear resolvent has at least three distinct degree four factors, this eliminates the possibility of 8T5, 8T11, 8T19, 8T20, 8T22 and 8T29. Lastly, the numerical examples in Table \ref{table5} which are searched and verified using Magma \cite{magma} show that each of the remaining possibilities does occur. 
\end{proof}

\begin{table}[h]
\caption{Possible Galois groups $G_8$ of irreducible polynomials $x^8+ax^6+bx^4+ax^2+1 \in \Q[x]$ based on the Galois group $G_4$ of $x^4+ax^3+bx^2+ax+1 \in \Q[x]$. \label{table3}}
{\begin{tabular}{|c|c|} \hline 
$\mathbf{G_4}$ & \textbf{Possible $\mathbf{G_8}$} \\ \hline
$E_4$ & 8T2, 8T3, 8T4, 8T5, 8T9, 8T11, 8T22 \\ \hline
$C_4$ & 8T2, 8T10, 8T20 \\ \hline
$D_4$ & 8T4, 8T9, 8T10, 8T18, 8T19, 8T29 \\ \hline
\end{tabular}}
\end{table}

\begin{table}[h]
\caption{Possible Galois groups $G$ of $x^8+ax^6+bx^4+ax^2+1 \in \Q[x]$ and list of orbit lengths for the action of $G$ on all 2-sets of eight letters. Multiplicative notation is used to denote multiplicity. \label{table4}}
{\begin{tabular}{|c|c|} \hline
$\mathbf{G}$ & \textbf{Orbits} \\ \hline
8T2 & $4^3,8^2$ \\ \hline
8T3 & $4^7$ \\ \hline
8T4 & $4^5,8$ \\ \hline
8T5 & $4,8^3$ \\ \hline
8T9 & $4^3,8^2$ \\ \hline
8T10 & $4^3,16$ \\ \hline
\end{tabular}
\begin{tabular}{|c|c|} \hline
$\mathbf{G}$ & \textbf{Orbits} \\ \hline
8T11 & $4,8^3$ \\ \hline
8T18 & $4^3,16$ \\ \hline
8T19 & $4,8,16$ \\ \hline
8T20 & $4,8,16$ \\ \hline
8T22 & $4,8^3$ \\ \hline
8T29 & $4,8,16$ \\ \hline
\end{tabular}}
\end{table}

\begin{table}[h]
\caption{Numerical examples of irreducible polynomials $x^8+ax^6+bx^4+ax^2+1 \in \Q[x]$ (where $a \neq 0$) with Galois group $H_8$, where $x^4+ax^3+bx^2+ax+1$ has Galois group $H_4$. \label{table5}}
{\begin{tabular}{|c|c|c|} \hline
$\mathbf{H_4}$ & $\mathbf{H_8}$ & \textbf{Polynomial} \\ \hline
$E_4$ & 8T2 & $x^8+24x^6+48x^4+24x^2+1$ \\ \hline
$E_4$ & 8T3 & $x^8-3x^6+8x^4-3x^2+1$ \\ \hline
$E_4$ & 8T4 & $x^8+4x^6+8x^4+4x^2+1$ \\ \hline 
$E_4$ & 8T9 & $x^8+2x^6-7x^4+2x^2+1$ \\ \hline
$C_4$ & 8T2 & $x^8-x^6+x^4-x^2+1$ \\ \hline
$C_4$ & 8T10 & $x^8+x^6-9x^4+x^2+1$ \\ \hline
$D_4$ & 8T4 & $x^8+x^6-3x^4+x^2+1$ \\ \hline
$D_4$ & 8T9 & $x^8+x^6+4x^4+x^2+1$ \\ \hline
$D_4$ & 8T10 & $x^8+4x^6-2x^4+4x^2+1$ \\ \hline
$D_4$ & 8T18 & $x^8+x^6-x^4+x^2+1$ \\ \hline 
\end{tabular}}
\end{table}

\subsection{Classification of Galois groups}

In this subsection, we completely classify the Galois groups of palindromic even octic polynomials for the cases where the Galois group of $g(x)=x^4+ax^3+bx^2+ax+1$ is either $E_4$ or $C_4$. We begin by providing a parameterization for the irreducible polynomials $x^4+ax^3+bx^2+ax+1 \in \Q[x]$ with Galois group $E_4$.

\begin{lemma}
Let $f(x)=x^4+ax^3+bx^2+ax+1 \in \Q[x]$ be an irreducible polynomial. Then $\Gal(f)$ is $E_4$ if and only if $a=pqt$ and $b=(p^2+q^2)t-2$ for some $p,q,t \in \Q$.
\end{lemma}

\begin{proof}
By Proposition \ref{pocticsgal}, $\Gal(f)$ is $E_4$ if and only if $(b+2)^2-4a^2=r^2$ for some $r \in \Q$. This can be rearranged into $r^2+(2a)^2=(b+2)^2$ which is the well-known Diophantine equation on Pythagorean triples, whose rational solutions are of the form $r=(p^2-q^2)t$, $2a=2pqt$ and $b+2=(p^2+q^2)t$ for some $p,q,t \in \Q$.
\end{proof}

Using this parameterization, $R_{16}(x)$ factors as $S_1(x^2)$ and $S_2(x^2)$ where
\begin{equation} \label{s}
\begin{split}
S_1(x)=x^4+2pqt x^3+(8+2p^2 t-4q^2 t+p^2 q^2 t^2)x^2+2pqt(p^2 t-4)x+(p^2 t-4)^2, \\
S_2(x)=x^4+2pqt x^3+(8+2q^2 t-4p^2 t+p^2 q^2 t^2)x^2+2pqt(q^2 t-4)x+(q^2 t-4)^2.
\end{split}
\end{equation}

We remark that $S_1(x)$ and $S_2(x)$ can be obtained from one another by permuting $p$ and $q$ since $a$ and $b$ are symmetrical functions of $p$ and $q$. In particular, $p^2 t$ and $q^2 t$ are invariants, which we can determine by solving the simultaneous equations $(p^2 t)(q^2 t)=a^2$ and $p^2 t+q^2 t=b+2$ to obtain
\begin{equation*}
p^2 t=\frac{1}{2}\left(b+2+\sqrt{(b+2)^2-4a^2}\right) \ \text{ and } \ q^2 t=\frac{1}{2}\left(b+2-\sqrt{(b+2)^2-4a^2}\right).
\end{equation*}

The following lemma on whether certain quantities are squares will be helpful in describing the factorization patterns of the degree eight factors $S_1(x^2)$ and $S_2(x^2)$.

\begin{lemma} \label{plem}
Let $x^4+ax^3+bx^2+ax+1 \in \Q[x]$ (where $a \neq 0$) be an irreducible polynomial with Galois group $E_4$. Using the parameterization above, we have
\begin{enumerate}
\item $a^2-4b+8 \notin \Q^2$.
\item $p^2 t \in \Q^2$ if and only if $q^2 t \in \Q^2$ if and only if $b+2+2a \in \Q^2$.
\item $p^2 t (q^2 t - 4) \notin \Q^2$ and $q^2 t (p^2 t - 4) \notin \Q^2$.
\end{enumerate}
\end{lemma}

\begin{proof}
\textbf{(1)} Suppose to the contrary that $a^2-4b+8=r^2$ for some $r \in \Q$. Then $b=(a^2+8-r^2)/4$ and hence,
\begin{equation*}
x^4+ax^3+bx^2+ax+1=\left(x^2+\frac{a+r}{2}x+1\right)\left(x^2+\frac{a-r}{2}x+1\right)
\end{equation*}
is reducible, a contradiction.
\par\textbf{(2)} Note that $b+2+2a=p^2 t+q^2 t+2pqt=(p+q)^2 t$ and hence, $b+2+2a \in \Q^2$ if and only if $t \in \Q^2$. Since $pqt=a \neq 0$, it follows that $p \neq 0$ and $q \neq 0$. Therefore, $t \in \Q^2$ if and only if $p^2 t \in \Q^2$. Similarly, we have $t \in \Q^2$ if and only if $q^2 t \in \Q^2$.
\par\textbf{(3)} Suppose to the contrary that $p^2 t (q^2 t - 4)=r^2$ for some $r \in \Q$. Note that
\begin{equation*}
p^2 t(q^2 t-4) = a^2 - 2b - 4 - 2\sqrt{(b+2)^2-4a^2}.
\end{equation*}
It follows that
\begin{equation}
(a^2-2(b+2)-r^2)^2 = 4\big((b+2)^2-4a^2\big). \label{f}
\end{equation}
Note that $a^2 \neq r^2$, otherwise we have $a=0$ from (\ref{f}). Therefore,
\begin{equation*}
b = \frac{a^4-2a^2 r^2+r^4+16a^2}{4(a^2-r^2)} - 2.
\end{equation*}
This implies that
\begin{equation*}
x^4+ax^3+bx^2+ax+1=\left(x^2+\frac{a+r}{2}x+\frac{a+r}{a-r}\right)\left(x^2+\frac{a-r}{2}x+\frac{a-r}{a+r}\right)
\end{equation*}
is reducible, a contradiction. Likewise, suppose to the contrary that $q^2 t (p^2 t - 4)=r^2$ for some $r \in \Q$. Then we also obtain (\ref{f}) and hence, a contradiction.
\end{proof}

We remark that any irreducible factor of $R(x)$ that occurs with multiplicity one has degree at least four. This can be verified by computing the list of orbit lengths of the action of $G$ on all the 2-sets of eight letters for each transitive subgroup $G$ of $S_8$. In particular, any square-free degree four factor of $R(x)$ is irreducible. Therefore, if $S_1(x^2)$ or $S_2(x^2)$ is reducible, then it necessarily factors as the product of two irreducible quartic polynomials. This occurs if and only if either $S_1(x)$ is reducible and factors as a product of two quadratic polynomials, or that $S_1(x)$ is irreducible and satisfies Proposition \ref{polycompeq}. We are now ready to completely characterize the irreducibility of $S_1(x^2)$ and $S_2(x^2)$. 

\begin{proposition} \label{linresfactor2}
Define $S_1,S_2$ as in {\upshape (\ref{s})}. Then
\begin{enumerate}
\item $S_1(x^2)$ is irreducible if and only if $b+2+2a \notin \Q^2$ and $\frac{b-6-\sqrt{(b+2)^2-4a^2}}{2} \notin \Q^2$.
\item $S_2(x^2)$ is irreducible if and only if $b+2+2a \notin \Q^2$ and $\frac{b-6+\sqrt{(b+2)^2-4a^2}}{2} \notin \Q^2$.
\end{enumerate}
Furthermore, we cannot have both $\frac{b-6+\sqrt{(b+2)^2-4a^2}}{2} \in \Q^2$ and $\frac{b-6-\sqrt{(b+2)^2-4a^2}}{2} \in \Q^2$. Therefore, both $S_1(x^2)$ and $S_2(x^2)$ are reducible if and only if $b+2+2a \in \Q^2$. Moreover, each of the irreducible factors in $S_1(x^2)$ and $S_2(x^2)$ occurs as a factor of $R(x)$ with multiplicity one.
\end{proposition}

\begin{proof}
Suppose that $S_1(x^2)$ is reducible. We first consider the case where $S_1(x)$ is reducible and factors as a product of two quadratic polynomials. It follows that $S_1(x-\frac{pqt}{2})$ is a quartic polynomial with no cubic term, and it also factors as a product of two quadratic polynomials. Then by Proposition \ref{quarticreducible}, we have at least one of $q^2 t$, $q^2 t - 4$ or $p^2 t (q^2 t - 4)$ is in $\Q^2 \setminus \{0\}$, or that $4pqt(q^2 t - 4)=0$ and $8(q^2 t - 4) (-2 p^2 t + 2 q^2 t + p^2 q^2 t^2) \in \Q^2$. Since $pqt=a \neq 0$, it follows that the last condition is achieved if and only if $q^2 t - 4 = 0$, that is, $p^2 t(q^2 t-4)=0 \in \Q^2$. By Lemma \ref{plem}, we have $p^2 t (q^2 t - 4) \notin \Q^2$. Therefore, we conclude that either $q^2 t \in \Q^2$ (equivalently, $b+2+2a \in \Q^2$), or $q^2 t - 4 \in \Q^2$ (equivalently, $\frac{b-6-\sqrt{(b+2)^2-4a^2}}{2} \in \Q^2$).
\par We now consider the case where $S_1(x)$ is irreducible. By Proposition \ref{polycompcor}, we have $n=\pm(p^2 t - 4)$. If $n=p^2 t - 4$, then the equation (\ref{ll}) has a rational solution of the form
\begin{equation*}
r=\pm \sqrt{q^2 t(p^2 t - 4)} \pm 2\sqrt{2}\sqrt{p^2 t -2 \mp \frac{pqt}{q^2 t}\sqrt{q^2 t(p^2 t - 4)}}.
\end{equation*}
This implies that
\begin{eqnarray*}
&& (r \mp \sqrt{q^2 t(p^2 t - 4)})^2 = \left( \pm 2\sqrt{2}\sqrt{p^2 t -2 \mp \frac{pqt}{q^2 t}\sqrt{q^2 t(p^2 t - 4)}} \right)^2 \\
&\implies& r^2 + q^2 t(p^2 t - 4) \mp 2r\sqrt{q^2 t(p^2 t - 4)} = 8 \left(p^2 t - 2 \mp \frac{pqt}{q^2 t}\sqrt{q^2 t(p^2 t - 4)}\right) \\
&\implies& r^2 + q^2 t(p^2 t - 4) + 16 - 8 p^2 t = \pm 2r\sqrt{q^2 t(p^2 t - 4)} \mp \frac{8pqt}{q^2 t}\sqrt{q^2 t(p^2 t - 4)} \\
&\implies& r^2 + q^2 t(p^2 t - 4) + 16 - 8 p^2 t = \pm \left( 2r - \frac{8pqt}{q^2 t} \right) \sqrt{q^2 t(p^2 t - 4)}.
\end{eqnarray*}
Now if $r=\frac{4pqt}{q^2 t}$, then
\begin{eqnarray*}
0 = r^2 + q^2 t(p^2 t - 4) + 16 - 8p^2 t = \frac{(q^2 t - 4)(p^2 q^2 t^2 - 4p^2 t - 4q^2 t)}{q^2 t}
\end{eqnarray*}
and so $q^2 t - 4 = 0$ or $p^2 q^2 t^2 - 4p^2 t - 4q^2 t=0$. The first condition implies that $p^2 t(q^2 t - 4)=0 \in \Q^2$, contradicting Lemma \ref{plem}. The second condition implies that $a^2-4b-8=0$, and so $a^2-4b+8=16 \in \Q^2$, contradicting Lemma \ref{plem}. It follows that $2r-\frac{8pqt}{q^2 t} \neq 0$ and hence, $\sqrt{q^2 t(p^2 t - 4)} \in \Q$, again contradicting Lemma \ref{plem}. If $n=-(p^2 t-4)$, then the equation (\ref{ll}) has solutions
\begin{equation*}
r=\pm 4 \pm \sqrt{a^2-4b+8} \notin \Q
\end{equation*}
since $a^2-4b+8 \notin \Q^2$ by Lemma \ref{plem}. Therefore, $S_1(x^2)$ is irreducible if and only if $S_1(x)$ is irreducible, and this proves $(1)$. A similar argument proves $(2)$. 
\par Suppose to the contrary that $\frac{b-6+\sqrt{(b+2)^2-4a^2}}{2} \in \Q^2$ and $\frac{b-6-\sqrt{(b+2)^2-4a^2}}{2} \in \Q^2$. Then their product is $a^2-4b+8 \in \Q^2$, contradicting Lemma \ref{plem}. Lastly, we know from Proposition \ref{pmult} that each irreducible factor in $S_1(x^2)$ and $S_2(x^2)$ is not identical to $R_1(x)$ and $R_2(x)$, and note the following:
\begin{itemize}
\item If $b+2+2a \in \Q^2$, then $S_1(x^2)=S_{11}(x)S_{12}(x)$ and $S_2(x^2)=S_{21}(x)S_{22}(x)$ where
\begin{equation*}
\begin{split}
S_{11}(x) = x^4 + (pqt+2\sqrt{q^2 t})x^2 + (2+\sgn(pqt)\sqrt{p^2 t})^2, \\
S_{12}(x) = x^4 + (pqt-2\sqrt{q^2 t})x^2 + (2-\sgn(pqt)\sqrt{p^2 t})^2, \\
S_{21}(x) = x^4 + (pqt+2\sqrt{p^2 t})x^2 + (2+\sgn(pqt)\sqrt{q^2 t})^2, \\
S_{22}(x) = x^4 + (pqt-2\sqrt{p^2 t})x^2 + (2-\sgn(pqt)\sqrt{q^2 t})^2. 
\end{split}
\end{equation*}
\item If $\frac{b-6+\sqrt{(b+2)^2-4a^2}}{2} \in \Q^2$, then $S_1(x^2)$ is irreducible and $S_2(x^2)=S_{21}(x)S_{22}(x)$ where
\begin{equation*}
\begin{split}
S_{21}(x) = x^4 + (pqt+2\sqrt{p^2 t-4})x^2 + (q^2 t-4), \\
S_{22}(x) = x^4 + (pqt-2\sqrt{p^2 t-4})x^2+ (q^2 t-4).
\end{split}
\end{equation*}
\item If $\frac{b-6-\sqrt{(b+2)^2-4a^2}}{2} \in \Q^2$, then $S_2(x^2)$ is irreducible and $S_1(x^2)=S_{11}(x)S_{12}(x)$ where
\begin{equation*}
\begin{split}
S_{11}(x) = x^4 + (pqt+2\sqrt{q^2 t-4})x^2 + (p^2 t-4), \\
S_{12}(x) = x^4 + (pqt-2\sqrt{q^2 t-4})x^2+ (p^2 t-4).
\end{split}
\end{equation*}
\end{itemize}
It is routine to verify that any two such irreducible factors are identical if and only if $p^2 t=q^2 t$, $p^2 t=0$, $q^2 t=0$, $p^2 t -4=0$ or $q^2 t -4=0$. The first possibility implies that $(b+2)^2-4a^2=0$, that is, $b=-2\pm 2a$ and hence,
\begin{equation*}
x^8+ax^6+bx^4+ax^2+1=(x^2\pm1)^2\left(x^4+(a\mp2)x^2+1\right),
\end{equation*}
a contradiction. The remaining possibilities imply that $p^2 t(q^2 t-4)=0 \in \Q^2$ or $q^2 t(p^2 t -4)=0 \in \Q^2$, contradicting Lemma \ref{plem}. This completes the proof of the proposition.
\end{proof}

We are now ready to describe an algorithm similar to Theorem \ref{deoctics} to completely classify the Galois groups of palindromic even octic polynomials $f(x)=x^8+ax^6+bx^4+ax^2+1$ (where $a \neq 0$) for the case where the Galois group of $x^4+ax^3+bx^2+ax+1$ is $E_4$ using the factorization patterns of the two degree eight factors $S_1(x^2)$ and $S_2(x^2)$. In particular, it follows from Table \ref{table4} that $\Gal(f)$ is 8T3 if both $S_1(x^2)$ and $S_2(x^2)$ are reducible, $\Gal(f)$ is 8T4 if exactly one of them is reducible, and $\Gal(f)$ is either 8T2 or 8T9 if none of them are reducible. To distinguish between 8T2 and 8T9, we first determine the three length four orbits for the action of $G \in \{\text{8T2},\text{8T9}\}$ on the 2-sets of eight letters. For each of these orbits $O$, we determine the image of the permutation representation of $G$ acting on $O$. These can be implemented in GAP \cite{gap} using the commands \verb|Orbits(G,Combinations([1..8],2),OnSets)| and \verb|Action(G,O,OnSets)|. For 8T2, the three orbits correspond to $E_4,C_4,C_4$ whereas for 8T9, the three orbits correspond to $E_4,D_4,D_4$. For each case, we will show that  both the degree four factors $R_1(x)$ and $R_2(x)$ cannot have Galois group $E_4$ and hence, must have the same Galois group (either $C_4$ or $D_4$). Since the degree four factors occur in $R(x)$ with multiplicity one by Proposition \ref{pmult}, it follows from Proposition \ref{soicher} that the Galois group of these factors corresponds to the repeating output. Therefore, we can distinguish between 8T2 and 8T9 by using Proposition \ref{quarticgal} to compute the Galois groups of $R_1(x)$ and $R_2(x)$.

\begin{theorem} \label{peocticse4} 
Let $f(x)=x^8+ax^6+bx^4+ax^2+1 \in \Q[x]$ (where $a \neq 0$) be an irreducible polynomial with $(b+2)^2-4a^2 \in \Q^2$. Then the following algorithm returns $\Gal(f)$.
\begin{enumerate}
\item If $b+2+2a \in \Q^2$, return 8T3 and terminate.
\item Else if 
\begin{equation*}
\frac{b-6+\sqrt{(b+2)^2-4a^2}}{2} \in \Q^2 \ \text{ or } \ \frac{b-6-\sqrt{(b+2)^2-4a^2}}{2} \in \Q^2, 
\end{equation*}
return 8T4 and terminate.
\item Else if $(a^2-4b+8)(b+2+2a) \in \Q^2$, return 8T2 and terminate.
\item Otherwise, return 8T9 and terminate.
\end{enumerate}
\end{theorem}

\begin{proof}
It follows from Proposition \ref{linresfactor2} that both $S_1(x^2)$ and $S_2(x^2)$ are reducible if and only if $b+2+2a \in \Q^2$, and exactly one of them is reducible if and only if either $\frac{b-6+\sqrt{(b+2)^2-4a^2}}{2} \in \Q^2$ or $\frac{b-6-\sqrt{(b+2)^2-4a^2}}{2} \in \Q^2$. This proves $(1)$ and $(2)$. Now if $b+2+2a \notin \Q^2$, then $b+2-2a=\frac{(b+2)^2-4a^2}{b+2+2a} \notin \Q^2$, and so both $\Gal(R_1)$ and $\Gal(R_2)$ are not $E_4$, and the Galois group of $R_2(x)$ is $C_4$ if $(a^2-4b+8)(b+2+2a) \in \Q^2$, and $D_4$ otherwise. This proves $(3)$ and $(4)$.
\end{proof}

For our final result, we consider the case where the Galois group of $x^4+ax^3+bx^2+ax+1$ is $C_4$.

\begin{theorem} \label{peocticsc4}
Let $f(x)=x^8+ax^6+bx^4+ax^2+1 \in \Q[x]$ (where $a \neq 0$) be an irreducible polynomial with $(b+2)^2-4a^2 \notin \Q^2$ and $(a^2-4b+8)\big((b+2)^2-4a^2) \in \Q^2$. Then $\Gal(f)$ is
\begin{enumerate}
\item 8T2 if $b+2-2a \in \Q^2$ or $b+2+2a \in \Q^2$,
\item 8T10 otherwise.
\end{enumerate}
\end{theorem}

\begin{proof} 
Let $g(x)=x^4+ax^3+bx^2+ax+1$ so that $f(x)=g(x^2)$. In view of Proposition \ref{pocticsgal}, we have $\Gal(g)$ is $C_4$, and it follows from Proposition \ref{peocticspossible} that $\Gal(f)$ is either 8T2 or 8T10. We first determine the three length four orbits for $G \in \{\text{8T2},\text{8T10}\}$ on the 2-sets of eight letters. For each of these orbits $O$, we determine the image of the permutation representation of $G$ acting on $O$. These can be implemented in GAP \cite{gap} using the commands \verb|Orbits(G,Combinations([1..8],2),OnSets)| and \verb|Action(G,O,OnSets)|. For 8T2, the three orbits correspond to $E_4,C_4,C_4$ whereas for 8T10, the three orbits correspond to $C_4,D_4,D_4$. Suppose to the contrary that both $\Gal(R_1)$ and $\Gal(R_2)$ are $C_4$. Then by Proposition \ref{quarticgal}, we have $(a^2-4b+8)(b+2-2a) \in \Q^2$ and $(a^2-4b+8)(b+2+2a) \in \Q^2$, respectively. It follows that their product $(a^2-4b+8)^2 \big((b+2)^2-4a^2\big) \in \Q^2$ and hence, $(b+2)^2-4a^2 \in \Q^2$, a contradiction. As such, $\Gal(f)$ is 8T2 if and only if $\{\Gal(R_1),\Gal(R_2)\}=\{E_4,C_4\}$. If $\Gal(f)$ is 8T10, then neither $\Gal(R_1)$ nor $\Gal(R_2)$ is $E_4$. By Proposition \ref{quarticgal}, $\Gal(R_1)$ is $E_4$ if and only if $b+2-2a \in \Q^2$, whereas $\Gal(R_2)$ is $E_4$ if and only if $b+2+2a \in \Q^2$. This completes the proof of the theorem.
\end{proof}

\appendix

\section{Complete Expression of $R_{16}(x)$} \label{aa}

$R_{16}(x)=x^{16}+a_{14}x^{14}+a_{12}x^{12}+a_{10}x^{10}+a_8x^8+a_6x^6+a_4x^4+a_2x^2+a_0$ where

\begin{eqnarray*}
a_{14}&=&4a, \\
a_{12}&=&2(6+3a^2-b), \\
a_{10}&=&2a(6+2a^2-b), \\
a_8&=&20+22a^2+a^4-52b+2a^2b-7b^2, \\
a_6&=&2a(-28+4a^2-4b+a^2b-3b^2), \\
a_4&=&192-32a^2+2a^4+16b-6a^2b+16b^2+a^2b^2-4b^3, \\
a_2&=&2a(8+a^2-4b)(-6+b), \\
a_0&=&(8+a^2-4b)^2.
\end{eqnarray*}

\section*{Acknowledgements}

The authors are indebted to the referee for the detailed comments and suggestions which helped to improve the manuscript considerably, especially on the classification of 8T11 and 8T22 in Theorem \ref{deoctics}.


\begin{bibdiv}
  \begin{biblist}

\bib{awtreyoctic}{article}{
	title = {Galois groups of doubly even octic polynomials},
	volume = {19},
	number = {1},
	journal = {J. Algebra Appl.},
	author = {Altmann, Anna},
	author = {Awtrey, Chad},
	author = {Cryan, Sam},
	author = {Shannon, Kiley},
	author = {Touchette, Madeleine},
	year = {2020},
	pages = {2050014}
}

\bib{awtreysubfield}{article}{
	title = {Subfields of solvable sextic field extensions},
	volume = {4},
	journal = {North Carolina J. Math. Stat.},
	author = {Awtrey, Chad},
	author = {Jakes, Peter},
	year = {2018},
	pages = {1--11}
}

\bib{awtreysextic}{article}{
	title = {Galois groups of even sextic polynomials},
	volume = {63},
	number = {3},
	journal = {Canad. Math. Bull.},
	author = {Awtrey, Chad},
	author = {Jakes, Peter},
	year = {2020},
	pages = {670--676}
}

\bib{magma}{article}{
    AUTHOR = {Bosma, Wieb},
   AUTHOR = {Cannon, John},
   AUTHOR = {Playoust, Catherine},
     TITLE = {The {M}agma algebra system. {I}. {T}he user language},
   JOURNAL = {J. Symbolic Comput.}, 
    VOLUME = {24},
      YEAR = {1997},
    NUMBER = {3--4},
     PAGES = {235--265},
}

\bib{brookfield}{article}{
	title = {Factoring quartic polynomials: A lost art},
	volume = {80},
	number = {1},
	journal = {Math. Mag.},
	author = {Brookfield, Gary},
	year = {2007},
	pages = {67--70}
}

\bib{butlermckay}{article}{
	title = {The transitive groups of degree up to eleven},
	volume = {11},
	number = {8},
	journal = {Comm. Algebra},
	author = {Butler, Gregory},
	author = {McKay, John},
	year = {1983},
	pages = {863--911}
}

\bib{dean}{article}{
	title = {A rational polynomial whose group is the quaternions},
	volume = {88},
	number = {1},
	journal = {Amer. Math. Monthly},
	author = {Dean, Richard A.},
	year = {1981},
	pages = {42--45}
}

\bib{polycomp1}{article}{
	title = {${A}_4$-sextic fields with a power basis},
	volume = {19},
	number = {3},
	journal = {Missouri J. Math. Sci.},
	author = {Eloff, Daniel},
	author = {Spearman, Blair K.},
	author = {Williams, Kenneth S.},
	year = {2007},
	pages = {188--194}
}

\bib{harringtonjones}{article}{
	title = {The irreducibility of power compositional sextic polynomials and their {Galois} groups},
	volume = {120},
	number = {2},
	journal = {Math. Scand.},
	author = {Harrington, Joshua},
	author = {Jones, Lenny},
	year = {2017},
	pages = {181--194}
}

\bib{polycomp2}{article}{
	title = {Infinite families of ${A}_4$-sextic polynomials},
	volume = {57},
	number = {3},
	journal = {Canad. Math. Bull.},
	author = {Ide, Joshua},
	author = {Jones, Lenny},
	year = {2014},
	pages = {538--545}
}

\bib{polycomp3}{article}{
	title = {An infinite family of ninth degree dihedral polynomials},
	volume = {97},
	number = {1},
	journal = {Bull. Aust. Math. Soc.},
	author = {Jones, Lenny},
	author = {Phillips, Tristan},
	year = {2018},
	pages = {47--53}
}

\bib{kaplansky}{book}{
	edition = {2nd ed.},
	title = {Fields and {Rings}},
	publisher = {The University of Chicago Press},
	author = {Kaplansky, Irving},
	year = {1972}
}

\bib{kappewarren}{article}{
	title = {An elementary test for the {Galois} group of a quartic polynomial},
	volume = {96},
	number = {2},
	journal = {Amer. Math. Monthly},
	author = {Kappe, Luise-Charlotte},
	author = {Warren, Bette},
	year = {1989},
	pages = {133--137}
}

\bib{polycomp4}{article}{
	title = {Lifting monogenic cubic fields to monogenic sextic fields},
	volume = {34},
	number = {3},
	journal = {Kodai Math. J.},
	author = {Lavallee, Melisa J.},
	author = {Spearman, Blair K.},
	author = {Williams, Kenneth S.},
	year = {2011},
	pages = {410--425}
}

\bib{schinzel}{book}{
	title = {Polynomials with special regard to reducibility},
	publisher = {Cambridge University Press},
	author = {Schinzel, Andrzej},
	year = {2000}
}

\bib{soicher}{thesis}{
	title = {The computation of {Galois} groups},
	school = {Concordia University},
	author = {Soicher, Leonard},
	year = {1981},
	type = {Master's thesis}
}

\bib{polycomp5}{article}{
	title = {The simplest ${D}_4$-octics},
	volume = {2},
	number = {2},
	journal = {Int. J. Algebra},
	author = {Spearman, Blair K.},
	author = {Williams, Kenneth S.},
	year = {2008},
	pages = {79--89}
}

\bib{stauduhar}{article}{
	title = {The determination of {Galois} groups},
	volume = {27},
	journal = {Math. Comp.},
	author = {Stauduhar, Richard P.},
	year = {1973},
	pages = {981--996}
}

\bib{gap}{webpage}{
    author       = {{The GAP~Group}},
    title        = {{GAP} -- {G}roups, {A}lgorithms, and {P}rogramming,
                    {V}ersion 4.11.1},
    year         = {2021},
    url = {https://www.gap-system.org}
}

\bib{lmfdb}{webpage}{
  author       = {{The LMFDB Collaboration}},
  title        = {The ${L}$-functions and modular forms database},
  url = {http://www.lmfdb.org},
  year         = {2021},
  note         = {[Online; accessed 12 November 2021]}
}

\bib{wmath}{webpage}{
  author = {{Wolfram Research, Inc.}},
  title = {Mathematica, {V}ersion 12.3.1},
  url = {https://www.wolfram.com/mathematica},
  year = {2021}
}

  \end{biblist}
\end{bibdiv}
\end{document}